\documentclass[11pt]{amsart}

\usepackage{amssymb}
\usepackage{verbatim}
\usepackage{hyperref}
\usepackage{color}

\newtheorem{theorem}{Theorem}[section]
\newtheorem{lemma}[theorem]{Lemma}
\newtheorem{corollary}[theorem]{Corollary}
\newtheorem{Proposition}[theorem]{Proposition}

\newtheorem*{thma}{Theorem A}
\newtheorem*{coroa}{Corollary A}
\newtheorem*{lemmaa}{Lemma A}
\newtheorem*{lemmab}{Lemma B}
\theoremstyle{remark}
\newtheorem{remark}[theorem]{Remark}
\theoremstyle{definition}
\newtheorem{definition}[theorem]{Definition}
\newtheorem{example}[theorem]{Example}

\numberwithin{equation}{section}

\def\R{{\mathbb R}}
\def\S{{\mathbb S}}

\def\intslash{\rlap{\kern  .32em $\mspace {.5mu}\backslash$ }\int}
\def\qsl{{\rlap{\kern  .32em $\mspace {.5mu}\backslash$ }\int_{Q_x}}}

\def\emph#1{{\it #1 }}

\def\rta{\rightarrow}

\def\pv{\text{\rm p.v.}}
\def\alp{\alpha}

\def\del{\delta}             
\def\eps{\varepsilon}
\def\ep{\epsilon}

\def\tet{\theta}

\def\lam{\lambda}

\def\si{\sigma}              
\def\vphi{\varphi}
             
\def\om{\omega}              \def\Om{\Omega}

\def\fr{\frac}
\newcommand{\Be}{\begin{equation}}
\newcommand{\Ee}{\end{equation}}
\newcommand{\Bes}{\begin{equation*}}
\newcommand{\Ees}{\end{equation*}}
\newcommand{\Bsp}{\begin{split}}
\newcommand{\Esp}{\end{split}}
\newcommand{\Bm}{\begin{multline}}
\newcommand{\Em}{\end{multline}}
\newcommand{\Bea}{\begin{eqnarray}}
\newcommand{\Eea}{\end{eqnarray}}
\newcommand{\Beas}{\begin{eqnarray*}}
\newcommand{\Eeas}{\end{eqnarray*}}
\newcommand{\Benu}{\begin{enumerate}}
\newcommand{\Eenu}{\end{enumerate}}
\newcommand{\Bi}{\begin{itemize}}
\newcommand{\Ei}{\end{itemize}}
\begin{document}

\title[$L^1$-Dini conditions and limiting behavior of weak type estimates]
{$L^1$-Dini conditions and limiting behavior of
 weak type estimates for singular integrals}

\author[Yong Ding]{Yong Ding}

\author[Xudong Lai]{Xudong Lai}

\address{\textbf{Yong Ding}\endgraf
 School of Mathematical Sciences\endgraf
         Beijing Normal University\endgraf
Laboratory of Mathematics and Complex Systems (BNU), Ministry of Education\endgraf
Beijing, 100875, P. R. of China\endgraf}
\email{dingy@bnu.edu.cn}
\thanks {The work is supported by NSFC (No.11371057, 11471033),  SRFDP (No.20130003110003) and the Fundamental Research Funds for the Central Universities (No.2014KJJCA10).}

\address{\textbf{Xudong Lai}(Corresponding Author)\endgraf
 School of Mathematical Sciences\endgraf
         Beijing Normal University\endgraf
Laboratory of Mathematics and Complex Systems (BNU), Ministry of Education\endgraf
Beijing, 100875, P. R. of China\endgraf}
\email{xudonglai@mail.bnu.edu.cn}
\thanks{Xudong Lai is the corresponding author.}



\begin{abstract}
In 2006, Janakiraman \cite{PJ} showed that
if $\Omega$  with mean value zero on $\S^{n-1}$ satisfies the condition:
$$ \sup_{|\xi|=1}\int_{\S^{n-1}}|\Omega(\theta)-\Omega(\theta+\del\xi)|d\si(\tet)\leq Cn\del\int_{\mathbb{S}^{n-1}}|\Om(\theta)|d\si(\theta),\quad(\ast)
$$
where $0<\del<\fr{1}{n}$,
then for the singular integral operator $T_\Om$ with homogeneous kernel, the following limiting behavior holds:
\[\lim\limits_{\lam\rta 0_+}\lam m(\{x\in\R^n:|T_\Om f(x)|>\lam\})= \fr{1}{n}\|\Om\|_{1}\|f\|_{1},\quad(\ast\ast)\]
for $f\in L^1(\R^n)$ with $f\geq 0$.
 
In the present paper, we prove that if replacing the condition $(\ast)$ by more general condition, the $L^1$-Dini condition, then
the limiting behavior $(\ast\ast)$ still holds for the singular integral $T_\Om$. In particular,
we give  an example which satisfies the $L^1$-Dini condition, but does not satisfy $(\ast)$.  Hence, we improve essentially the above result given in \cite{PJ}. To prove our conclusion, we show that the $L^1$-Dini conditions defined respectively via rotation and translation in $\R^n$ are equivalent (see Theorem \ref{t:5equ} below),
which has its own interest in the theory of singular integrals. Moreover, similar limiting behavior for the fractional integral operator $T_{\Om,\alpha}$ with homogeneous kernel is also established in this paper.
\end{abstract}

\maketitle

\section{Introduction}

Suppose that the function $\Om$ defined on $\R^n\setminus\{0\}$ satisfies the following conditions:
\Be\label{e:Mo1}
\Om(\lam x)=\Om(x),\quad \text{for any}\ \lam>0\ \text{and}\ x\in\R^n\setminus\{0\},
\Ee
\Be\label{e:Mo2}
\int_{\mathbb{S}^{n-1}}\Om(\tet)d\si(\tet)=0
\Ee
and $\Om\in L^1(\S^{n-1})$, where $\S^{n-1}$ denotes the unit sphere in $\R^n$ and $d\si$ is the area measure on $\S^{n-1}$. Then the singular integral $T_\Om$ with homogenous kernel is defined by
$$T_\Om f(x)=\pv\int\fr{\Om(x-y)}{|x-y|^n}f(y)dy.$$
It is well know that
if $\Om$ is odd and $\Om\in L^1(\S^{n-1})$ (or $\Om$ is even and
$\Om\in L\log^+L(\S^{n-1})$), $T_\Om$ is bounded on $L^p(\R^n)$ for $1<p<\infty$ (see \cite{CZ56}), that is,
\Be\label{e:5TOM}
\|T_\Om f\|_p\leq C_p\|f\|_p.
\Ee
For $p=1$, Seeger \cite{S1} showed
that if $\Om\in L\log^+L(\S^{n-1})$,
\Be\label{e:5WT}
m(\{x\in\R^n:|T_\Om f(x)|>\lam\})\leq C_1\fr{\|f\|_1}{\lam}.
\Ee
If $\Om$ is an odd function, the usual Calder\'on-Zygmund method of rotation gives some information of the constant in \eqref{e:5TOM}. In fact,
 $C_p=\fr{\pi}{2}H_p\|\Om\|_1$ (see \cite{IM}), where $H_p$ denotes the $L^p$ norm of the Hilbert transform ($1<p<\infty$).

In 2004, Janakiraman \cite{PJ_} showed
that the constants $C_p$ in \eqref{e:5TOM} and $C_1$ in \eqref{e:5WT} are at worst $C\log n\|\Om\|_1$ if $\Om$ satisfies \eqref{e:Mo1}, \eqref{e:Mo2} and the following  \emph{regularity condition}:
\begin{equation}\label{e:omega}
\sup_{|\xi|=1}\int_{\S^{n-1}}|\Omega(\theta)-\Omega(\theta+\del\xi)|d\si(\tet)\leq Cn\del\int_{\mathbb{S}^{n-1}}|\Om(\theta)|d\si(\theta),\quad 0<\del<\fr{1}{n},
\end{equation}
where $C$ is a constant independent of the dimension.  In 2006, Janakiraman \cite{PJ} extended further
this result to the limiting case.  Let $\mu$ be a signed measure on $\R^n$, which is absolutely continuous with respect to Lebesgue measure and
$|\mu|(\R^n)<\infty$, here $|\mu|$ is the total variation of $\mu$.
Define
\begin{equation}\label{singular with mea}
T_\Om\mu(x)=\pv\int\fr{\Om(x-y)}{|x-y|^n}d\mu(y).
\end{equation}
\begin{thma}[\cite{PJ}]
Suppose $\Om$ satisfies \eqref{e:Mo1}, \eqref{e:Mo2} and the regularity condition \eqref{e:omega}. Then
$$\lim\limits_{\lam\rta0_+}\lam m(\{x\in\R^n:|T_\Om\mu(x)|>\lam\})=\fr{1}{n}\|\Om\|_1|\mu(\R^n)|.$$
\end{thma}
As a consequence of Theorem A, Janakiraman showed indeed that
\begin{coroa}[\cite{PJ}]
Let $f\in L^{1}(\R^n)$ and $f\geq0$. Suppose $\Om$ satisfies \eqref{e:Mo1}, \eqref{e:Mo2} and  \eqref{e:omega}, then
\begin{equation}\label{limit}
\lim\limits_{\lam\rta 0_+}\lam m(\{x\in\R^n:|T_\Om f(x)|>\lam\})=\fr{1}{n}\|\Om\|_{1}\|f\|_{1}.
\end{equation}
\end{coroa}

The limiting behavior \eqref{limit} is very interesting since it gives some information of the best constant for weak type (1,1) estimate of the homogeneous singular integral operator $T_\Om$ in some sense. However, note that the condition \eqref{e:omega} seems to be strong compared with the \emph{H$\ddot{o}$rmander condition} (see also \cite{SWb}):
\begin{equation}\label{Hormand}
\sup\limits_{y\neq0}\int_{|x|>2|y|}|K(x-y)-K(x)|dx<\infty,
\end{equation}
where $K$ is the kernel of the Calder\'on-Zygmund singular integral operator.
Hence, it is natural to ask whether \eqref{limit} still holds if replacing  \eqref{e:omega} by the H$\ddot{o}$rmander condition \eqref{Hormand}?
The purpose of this paper is to give an affirmative answer to the above problem for the case of $K(x)=\Om(x)|x|^{-n}$.

Before stating our results, we give the definition of the $L^1$-Dini condition.

\begin{definition}[$L^1$-Dini condition]\label{d:5L1}
Let $\Om$ satisfy \eqref{e:Mo1}. We say that $\Om$ satisfies the $L^1$-Dini condition if:

(i)\ $\Om\in L^1(\S^{n-1})$;

(ii)\ $\int_0^1\fr{\om_1(\del)}{\del}d\del<\infty$, where $\om_1$ denotes the $L^1$ integral
modulus of continuity of $\Om$ defined by
$$\om_1(\del)=\sup\limits_{\|\rho\|\leq\del}\int_{\S^{n-1}}|\Om(\rho \tet)-\Om(\tet)|d\si(\tet),$$
where $\rho$ is a rotation on $\R^n$ and $\|\rho\|:=\sup\{|\rho x'-x'|:x'\in\S^{n-1}\}$.
\end{definition}

Let us recall two important facts in \cite{CWZ} and \cite{CZ79}.

\begin{lemmaa}[\cite{CWZ}]\label{CWZ}
If $\Om$ satisfies the $L^1$-Dini condition, then $\Om\in L\log^+\!\!L(\S^{n-1})$ and $K(x)=\Om(x)|x|^{-n}$ satisfies the H$\ddot{o}$rmander condition \eqref{Hormand}.
\end{lemmaa}

\begin{lemmab}[\cite{CZ79}]\label{CZ}
If $K(x)=\Om(x)|x|^{-n}$ satisfies the H$\ddot{o}$rmander condition \eqref{Hormand}, then $\Om\in L\log^+\!\!L(\S^{n-1})$ and $\Om$ satisfies the $L^1$-Dini condition.
\end{lemmab}

By Lemma A and Lemma B, one can see immediately that
for the kernel $K(x)=\Om(x)|x|^{-n}$ the H$\ddot{o}$rmander condition \eqref{Hormand} is equivalent to the $L^1$-Dini condition.

In Section \ref{s:52}, we will prove that the regularity condition \eqref{e:omega} is stronger than the $L^1$-Dini condition (see Proposition  \ref{l:3ker}).
Also we will give an example to show that the $L^1$-Dini condition is
rigorously weaker than the regularity condition \eqref{e:omega}
 (see Example \ref{exa:3}).

Our main result in this paper is to prove that the limiting behavior \eqref{limit} still holds if replacing the condition \eqref{e:omega}
by the $L^1$-Dini condition.

\begin{theorem}\label{t:main}
Suppose $\Om$ satisfies \eqref{e:Mo1}, \eqref{e:Mo2} and the $L^1$-Dini condition. Let $\mu$ be an absolutely continuous signed measure
 on $\R^n$ with respect to Lebesgue measure and $|\mu|(\R^n)<\infty$.
 Let $T_\Omega$ be defined by \eqref{singular with mea}.
 Then we have
\Be\label{e:main}
\lim\limits_{\lam\rta 0_+}\lam m(\{x\in\R^n:|T_\Om\mu(x)|>\lam\})= \fr{1}{n}\|\Om\|_{1}|\mu(\R^n)|.
\Ee

\end{theorem}

By setting $\mu(E)=\int_{E}f(x)dx$ with $f\in L^1(\R^n)$ in Theorem \ref{t:main}, we have the following result.
\begin{corollary}

Let $f\in L^{1}(\R^n)$ and $f\geq0$. Suppose $\Om$ satisfies \eqref{e:Mo1}, \eqref{e:Mo2} and the $L^1$-Dini condition. Then we have
\Bes
\lim\limits_{\lam\rta 0_+}\lam m(\{x\in\R^n:|T_\Om f(x)|>\lam\})= \fr{1}{n}\|\Om\|_{1}\|f\|_{1}.
\Ees
\end{corollary}

The next results are related to the limiting behavior for  weak type estimate of the homogenous fractional integral operator
 $T_{\Om,\alp}$, which is defined as
$$T_{\Om,\alp}f(x)=\int\fr{\Om(x-y)}{|x-y|^{n-\alp}}f(y)dy,\quad 0<\alpha<n.$$
It is well known that the fractional integral operator $T_{\Om,\alp}$,
a generalization of Riesz potential, has been studied by
many people (see the book \cite{LDY} and the references therein).
In \cite{DL}, while studying the boundedness of $T_{\Om,\alp}$ on Hardy space, Ding and Lu introduced the following regularity condition of $\Om$:
\Be\label{e:5frac}
\int_0^1\fr{\om_q(\del)}{\del^{1+\alp}}d\del<\infty,
\Ee
where $\om_q$ denotes the $L^q$ integral
modulus of continuity of $\Om$.

To study the limiting behavior of the fractional operator with homogeneous kernel, we need some regularity conditions similar to \eqref{e:5frac}. For convenience, we give the following notation.
\begin{definition}[$L^{s}_{\alp}$-Dini condition]\label{d:5Lalp}
 Let $\Om$ satisfy \eqref{e:Mo1}, $1\le s\le\infty$ and $0<\alp<n$. We say that $\Om$ satisfies the $L^{s}_{\alp}$-Dini condition if

(i)\ $\Om\in L^{s}(\S^{n-1})$;

(ii)\ $\int_0^1\fr{\om_1(\del)}{\del^{1+\alp}}d\del<\infty$, where
$\om_1$ is defined as that in Definition \ref{d:5L1}.

\end{definition}

Let $\nu$ be an absolutely continuous signed measure on $\R^n$
with respect to Lebesgue measure and $|\nu|(\R^n)<\infty$. Define
\Be\label{frac with mea}
T_{\Om,\alp}\nu(x)=\int\fr{\Om(x-y)}{|x-y|^{n-\alp}}d\nu(y).
\Ee

We have the following results for $T_{\Om,\alp}$.
\begin{theorem}\label{t:2}
Let $\nu$ be an absolutely continuous signed measure on $\R^n$
with respect to Lebesgue measure and $|\nu|(\R^n)<\infty$. Let $0<\alp<n$ and $r=\fr{n}{n-\alp}$. Suppose $\Om$ satisfies \eqref{e:Mo1}, \eqref{e:Mo2} and the $L^r_\alp$-Dini condition. Then
\Bes
\lim\limits_{\lam\rta 0_+}\lam^{r} m(\{x\in\R^n:|T_{\Om,\alp} \nu(x)|>\lam\})= \fr{1}{n}\|\Om\|^{r}_{r}|\nu(\R^n)|^r.
\Ees
\end{theorem}
\begin{corollary}
Let $0<\alp<n$ and $r=\fr{n}{n-\alp}$. Let $f\in L^{1}(\R^n)$ and $f\geq0$. Suppose $\Om$ satisfies \eqref{e:Mo1}, \eqref{e:Mo2} and the $L^r_\alp$-Dini condition. Then we have
\Bes
\lim\limits_{\lam\rta 0_+}\lam^{r} m(\{x\in\R^n:|T_{\Om,\alp} f(x)|>\lam\})= \fr{1}{n}\|\Om\|^{r}_{r}\|f\|_{1}^r.
\Ees
\end{corollary}

We would like to point out the proof of Theorem \ref{t:main} follows the idea from \cite{PJ}. However, to establish the limiting behavior of
the singular integral operator $T_\Om$ with $\Om$ satisfying the $L^1$-Dini condition,
we need study carefully the regularity of $\Om$. More precisely, we will
show that two different $L^1$-Dini conditions are equivalent (see Theorem \ref{t:5equ}).

The paper is organized as follows. In Section \ref{s:52}, we give some properties of the $L^1$-Dini condition and the embedding relation between the regularity condition \eqref{e:omega} and the $L^1$-Dini condition.
An example which shows the $L^1$-Dini condition is weaker than the condition \eqref{e:omega} is also given in this section.
The proof of Theorem \ref{t:main} is given in Section \ref{5sec:main}.
The outline of the proof of Theorem \ref{t:2} is given in final section. Throughout this paper the letter $C$ will stand for a positive constant not necessarily the same one in each occurrence.

\section{$L^1$-Dini condition}\label{s:52}

In this section, we discuss some properties of the $L^1$-Dini condition. We first show that the regularity condition \eqref{e:omega} is stronger than the $L^1$-Dini condition.

\begin{Proposition}\label{l:3ker}
If $\Om$ satisfies \eqref{e:Mo1}, \eqref{e:Mo2} and the condition \eqref{e:omega}, then $\Om$ satisfies $L^1$-Dini condition.
\end{Proposition}
\begin{proof}
We first claim that if $\Om$ satisfies \eqref{e:Mo1}, \eqref{e:Mo2} and  \eqref{e:omega}, then there exists $C>0$ such that
\begin{equation}\label{e:3om}
\om_1(\del)\leq\sup_{|\xi|=1}\int_{\S^{n-1}}|\Om(\tet+C\del\xi)-\Om(\tet)|d\tet
\end{equation}
for any $0<\del<\fr{1}{2}$.
To prove (\ref{e:3om}), by Definition \ref{d:5L1}, it is enough to show that for any fixed $\tet\in \S^{n-1}$,
$$\{\rho\tet:\|\rho\|\leq\del\}\subset\Big\{\fr{\tet+C\del\xi}
{|\tet+C\del\xi|}:\ \xi\in \S^{n-1}\Big\}$$
for some constant $C>0$. For convenience, let
$$A=\{\rho\tet:\|\rho\|\leq\del\}$$
and
$$B(C)=\Big\{\fr{\tet+C\del\xi}
{|\tet+C\del\xi|}:\ \xi\in \S^{n-1}\Big\}.$$
Thus  $A=\{\eta\in \S^{n-1}:\ |\eta-\tet|\leq\del\}$.
Choose $C=2$, we will show that
\begin{equation}\label{includ}
B(2)\supset A.\end{equation}
Notice that the function $f(\xi)=\Big|\fr{\tet+2\del\xi}{|\tet+2\del\xi|}-\tet\Big|$ is continuous on $\S^{n-1}$. Since $\S^{n-1}$ is compact, then $f(\xi)$ can get its maximal value at a point of $\S^{n-1}$. Suppose $\xi_0$ is such a point that $f(\xi)$ get its maximal value at $\xi_0$.
Since $f(\tet)=0$ and $f(-\tet)=0$, $\xi_0$  must be located between $\tet$ and $-\tet$.
Therefore again by the continuity of $f(\xi)$,
 $$B(2)=\{\eta\in \S^{n-1}:\ |\eta-\tet|\leq \gamma\}\quad \text{with} \quad \gamma=f(\xi_0).$$
So to prove \eqref{includ}, it suffices to show that $\gamma\geq\del$. By rotation, we may suppose $\tet=(1,0,0,\cdots,0)$. Choose $\xi=(0,1,0,\cdots,0)$. Then
$$\gamma\geq\Big|\fr{\tet+2\del\xi}{|\tet+2\del\xi|}-\tet\Big|=\Big(2-\fr{2}{\sqrt{1+4\del^2}}\Big)^{\fr{1}{2}}\geq\del.$$
Hence we prove (\ref{e:3om}) by choosing $C=2$.

Now we split the integral $\int_0^1\fr{\om_1(\del)}{\del}d\del$ into two parts:
$$\int_0^{\fr{1}{2n}}\fr{\om_1(\del)}{\del}d\del+\int_{\fr{1}{2n}}^{1}\fr{\om_1(\del)}{\del}d\del.$$
For the first integral, using estimate (\ref{e:3om}) and the regularity condition (\ref{e:omega}), we can get the bound $C\|\Om\|_1$. For the second integral, using $\om_1(\del)\leq2\|\Om\|_1$ for any $0<\del<1$, we can also get the bound $C\|\Om\|_1$. Combining these, the proof is completed.
\end{proof}
In the following, we give an example which satisfies \eqref{e:Mo1}, \eqref{e:Mo2} and the $L^1$-Dini condition but does not satisfy the regularity condition \eqref{e:omega}.
\begin{example}\label{exa:3}
Consider dimension $n=2$, in this case, we denote $\S^1=\{\theta:\ 0\leq\tet\le 2\pi\}$, where $\tet$ is the arc length on the unit circle. Let $\Om(\tet)=\tet^{-\fr{1}{2}}-(\frac 2{\pi})^{\fr{1}{2}}$. It can be easily extended to the whole space $\R^2$ such that
$\Om$ is homogeneous of degree zero.

By using representation of the differential of arc length, the integral of $\Om$ on $\S^1$ can be rewritten as
$$\int_0^{2\pi}\Om(\tet)d\tet,$$
where $\tet$ is again the arc length. Obviously, $\Om$ in Example \ref{exa:3} satisfies \eqref{e:Mo2}.

Now  let us first show that $\Om$ in Example \ref{exa:3} does not satisfy
the regularity condition ($\ref{e:omega}$). In fact, let $\del$ be small enough. In two dimension, for any rotation $\|\rho\|\leq\delta$, we have $\rho\tet=\tet\pm s$, where $s=\|\rho\|$.
For the case $\rho\tet=\tet+s$, we have
\begin{equation*}
\begin{split}
 \int_0^{2\pi}|\Om(\rho\tet)-\Om(\tet)|d\tet
&=\int_0^{2\pi-s}\bigg|\fr{1}{\tet^{1/2}}-\fr{1}{(\tet+s)^{1/2}}\bigg|d\tet\\
&\ \ \ \ +\int_{2\pi-s}^{2\pi}\bigg|\fr{1}{\tet^{1/2}}
-\fr{1}{(\tet+s-2\pi)^{1/2}}\bigg|d\tet\\
&=4((2\pi-s)^{1/2}-(2\pi)^{1/2}+s^{1/2})=:g(s),
\end{split}
\end{equation*}
where in the first equality we use the fact that when $\tet\in(2\pi-s,2\pi)$, $\rho\tet$ falls into $(0,s)$.
A similar computation shows that if $\rho\tet=\tet-s$,
\begin{equation*}
\int_0^{2\pi}|\Om(\rho\tet)-\Om(\tet)|d\tet=g(s).
\end{equation*}
It is not difficult to see that $g(s)$ is an increased function for $s\in[0,\del]$ and $g(0)=0$. Therefore we have
$$\om_1(\del)=\sup_{\|\rho\|\leq\del}\int_0^{2\pi}|\Om(\rho\tet)-\Om(\tet)|d\tet=g(\del).$$
Now by \eqref{e:3om} in Lemma \ref{l:3ker} (note that constant $C=2$), we have
\Bes
\begin{split}
\fr1{2\del}\sup_{|\xi|=1}\int_{\S^1}&|\Om(\tet+2\del\xi)
-\Om(\tet)|d\tet\geq
\fr1{2\del}\om_1(\del)\\
&=2\Big(\fr{1}{\del^{1/2}}-\fr{(2\pi)^{1/2}-(2\pi-\del)^{1/2}}{\del}\Big)\rta+\infty
\end{split}
\Ees
as $\del\rta0$. This means that $\Om$ does not satisfy the regularity condition \eqref{e:omega}.
By a direct computation, we have
$$\int_0^1\fr{\om_1(\del)}{\del}d\del=4\int_0^1\Big(\fr{1}{\del^{1/2}}-\fr{(2\pi)^{1/2}-(2\pi-\del)^{1/2}}{\del}\Big)d\del<\infty$$
and
$$\int_0^{2\pi}|\Om(\tet)|d\tet<\infty$$
which means that $\Om$ satisfies the $L^1$-Dini condition in Definition
\ref{d:5L1}.
\end{example}

In order to prove Theorem \ref{t:main}, we need to give an equivalent definition of the $L^1$-Dini condition in Definition
\ref{d:5L1}.

Recall in Definition \ref{d:5L1}, the $L^1$-Dini condition is defined by
the $L^1$ integral
modulus $\om_1$, and $\om_1$ is defined by  ROTATION in $\R^n$.    In \cite{CWZ}, Calder\'on, Weiss and Zygmund gave
another $L^1$ integral
modulus $\tilde{\om}_1$ which is defined by TRANSLATION in $\R^n$. Let $\Om$ satisfy \eqref{e:Mo1} and $\Om\in L^1(\S^{n-1})$.
Define $\tilde{\om}_1$ as
\begin{equation}\label{defin2}
\tilde{\om}_1(\del)=\sup\limits_{|h|\leq\del}\int_{\S^{n-1}}
|\Om(x'+h)-\Om(x')|d\si(x'),
\end{equation}
where $h\in\R^n$. Similarly, one may define the $L^1$-Dini condition by the
$L^1$ integral modulus $\tilde{\om}_1$.
\begin{definition}\label{d:5L11}
Let $\Om$ satisfy \eqref{e:Mo1}.  It is said that $\Om$ satisfies the $L^1$-Dini condition if:

(i)\ $\Om\in L^1(\S^{n-1})$;

(ii)\ $\int_0^1\fr{\tilde{\om}_1(\del)}{\del}d\del<\infty$, where $\tilde{\om}_1(\del)$ is defined by \eqref{defin2}.
\end{definition}

As it is pointed out in \cite{CWZ}, the $L^1$-Dini condition in Definition
\ref{d:5L1} is the most natural one. However, in some cases, the $L^1$-Dini definition in Definition \ref{d:5L11} is more convenient in application.
Thus, a natural problem is that, is there any relationship between those two kind of $L^1$-Dini conditions defined by
Definition \ref{d:5L1} and Definition \ref{d:5L11}, respectively.

Below we will show that these two kind $L^1$-Dini conditions defined respectively by Definition \ref{d:5L1} and Definition \ref{d:5L11} are equivalent indeed.
Let us first recall a useful lemma.
\begin{lemma}[see Lemma 5 in \cite{CWZ}]\label{l:3omesti}
There exist positive constants $\alp_0$, $C$ depending only on the dimension
$n$ such that if $\Om$ is any  function integrable over $\S^{n-1}$ and $0<|h|\leq\alp_0$, $h\in\R^n$, then
\begin{equation}
\int_{\S^{n-1}}|\Om(\xi-h)-\Om(\xi)|d\si(\xi)\leq C\sup_{\|\rho\|\leq|h|}\int_{\S^{n-1}}|\Om(\rho\xi)-\Om(\xi)|d\si(\xi).
\end{equation}
\end{lemma}
Note that we may choose the constant $\alp_0$ in Lemma \ref{l:3omesti} less than $1$.
\begin{theorem}\label{t:5equ}
$L^1$-Dini conditions defined respectively in Definition \ref{d:5L1} and Definition \ref{d:5L11} are equivalent.
\end{theorem}
\begin{proof}
By Definition \ref{d:5L1} and Definition \ref{d:5L11}, it is enough to show that for $\Om\in L^1(\S^{n-1})$, the following condition (a) and (b)
are equivalent:

(a)\  $\int_0^1\fr{\om_1(\del)}{\del}d\si(\del)<\infty$, where $\om_1(\del)=\sup\limits_{\|\rho\|\leq\del}\int_{\S^{n-1}}|\Om(\rho x')-\Om(x')|d\si(x')$,

(b)\ $\int_0^1\fr{\tilde{\om}_1(\del)}{\del}d\si(\del)<\infty$, where $\tilde{\om}_1(\del)=\sup\limits_{|h|\leq\del}\int_{\S^{n-1}}|\Om(x'+h)-\Om(x')|d\si(x')$.

We first show that (b) implies (a). By \eqref{e:3om} (note that the constant $C=2$), we have
$$\om_1(\del)\leq\sup\limits_{|\xi|=1}\int_{\S^{n-1}}|\Om(\tet+2\del\xi)-
\Om(\tet)|d\si(\tet)\leq\tilde{\om}_1(2\del).$$
Hence we have
\begin{equation*}
\begin{split}
\int_0^1\fr{\om_1(\del)}{\del}d\del&=\Big(\int_0^{1/2}+\int_{1/2}^1\Big)
\fr{\om_1(\del)}{\del}d\del\leq\int_0^{1/2}\fr{\tilde{\om}_1(2\del)}{\del}
d\del+\int_{1/2}^1\fr{\om_1(\del)}{\del}d\del\\
&\leq\int_0^1\fr{\tilde{\om}_1(\del)}{\del}d\del+C\|\Om\|_1.
\end{split}
\end{equation*}

Now we turn to the other part: (a) implies (b). By Lemma \ref{l:3omesti}, there exists a constant $0<a_0<1$ such that for any $0<|h|\leq a_0$,
we have
$$\int_{\S^{n-1}}|\Om(\xi+h)-\Om(\xi)|d\si(\xi)\leq C\sup_{\|\rho\|\leq|h|}\int_{\S^{n-1}}|\Om(\rho\tet)-\Om(\rho)|d\si(\tet).$$
If $0<\del<a_0$, then
\begin{equation*}
\begin{split}
 \tilde{\om}_1(\del)&=\sup\limits_{|h|\leq\del}\int_{\S^{n-1}}
 |\Om(\xi+h)-\Om(\xi)|d\si(\xi)\\
&\leq C\sup\limits_{|h|\leq\del}\sup_{\|\rho\|\leq|h|}\int_{\S^{n-1}}
|\Om(\rho\tet)-\Om(\tet)|d\si(\tet)\leq C \om_1(\del).
\end{split}
\end{equation*}
If $a_0\leq\del<1$, we get
$$\tilde{\om}_1(\del)=\sup\limits_{|h|\leq\del}\int_{\S^{n-1}}|\Om(\tet+h)-\Om(\tet)|d\si(\tet)
\leq\|\Om\|_1+\sup\limits_{|h|\leq\del}\int_{\S^{n-1}}|\Om(\tet+h)|d\si(\tet).$$
If we can prove that
\Be\label{e:5Omb}
\sup\limits_{|h|\leq\del}\int_{\S^{n-1}}|\Om(\tet+h)|d\si(\tet)\leq C\|\Om\|_1,
\Ee
then we have
\begin{equation*}
\begin{split}
 \int_0^1\fr{\tilde{\om}_1(\del)}{\del}d\del&=\Big(\int_0^{a_0}+\int_{a_0}^1\Big)
\fr{\tilde{\om}_1(\del)}{\del}d\del
\leq C\int_0^{1}\fr{\om_1(\del)}{\del}d\del+\int_{a_0}^1
\fr{\tilde{\om}_1(\del)}{\del}d\del\\
&\leq C\int_0^{1}\fr{\om_1(\del)}{\del}d\del+\int_{a_0}^1
\frac 1\delta\bigg(\|\Om\|_1+\sup\limits_{|h|\leq\del}
\int_{\S^{n-1}}|\Om(\tet+h)|d\si(\tet)\bigg)d\del\\
&\leq C\int_0^1\fr{{\om}_1(\del)}{\del}d\del+C\|\Om\|_1.
\end{split}
\end{equation*}
Hence, to complete the proof of Theorem \ref{t:5equ}, it remains to verify  \eqref{e:5Omb}.
By rotation, we may assume that $h=(h_1,0,\cdots,0)$, where $0<h_1<1$. By using the spherical coordinate formula on $\S^{n-1}$(see Appendix D in \cite{L}), we can write
\Be\label{e:3omb}
\begin{split}
\int_{\S^{n-1}}\Big|\Om\Big(\fr{x+h}{|x+h|}\Big)\Big|d\si(x)&=\int_{\vphi_1=0}^{\pi}\cdots\int_{\vphi_{n-2}=0}^{\pi}\int_{\vphi_{n-1}=0}^{2\pi}
\Big|\Om\Big(\fr{x(\vphi)+h}{|x(\vphi)+h|}\Big)\Big|\\
&\ \ \ \ \times|J(n,\vphi)|d\vphi_{n-1}\cdots d\vphi_1,
\end{split}
\Ee
where $x(\vphi)$ and $J(n,\vphi)$ are defined as
\Bes
\begin{split}
x_1&=\cos\vphi_1,\\
x_2&=\sin\vphi_1\cos\vphi_2,\\
x_3&=\sin\vphi_1\sin\vphi_2\cos\vphi_3,\\
&\cdots\\
x_{n-1}&=\sin\vphi_1\sin\vphi_2\cdots\sin\vphi_{n-2}\cos\vphi_{n-1},\\
x_n&=\sin\vphi_1\sin\vphi_2\cdots\sin\vphi_{n-2}\sin\vphi_{n-1};
\end{split}
\Ees
$$J(n,\vphi)=(\sin\vphi_1)^{n-2}\cdots(\sin\vphi_{n-3})^2\sin\vphi_{n-2}.$$
Compared with $x(\vphi)$, $\fr{x(\vphi)+h}{|x(\vphi)+h|}$ can be written as $x(\tet)$ with $\tet_i=\vphi_i, 2\leq i\leq n-1$. This can be seen from the point of geometry since $h=(h_1,0,\cdots,0)$. Hence we make a variable transform that maps $(\vphi_1,\vphi_2,\cdots,\vphi_{n-1})$ into $(\tet_1,\tet_2,\cdots,\tet_{n-1})$ such that
\Bes
\begin{cases}
\fr{\cos\vphi_1+h_1}{\sqrt{1+2h_1\cos\vphi_1+h_1^2}}&=\cos\tet_1,\ \fr{\sin\vphi_1}{\sqrt{1+2h_1\cos\vphi_1+h_1^2}}=\sin\tet_1,\\
\ \ \ \ \ \ \vphi_2&=\tet_2,\\
&\cdots\\
\ \ \ \ \ \ \vphi_{n-1}&=\tet_{n-1}.
\end{cases}
\Ees
Thus $\fr{x(\vphi)+h}{|x(\vphi)+h|}=x(\tet)$. It is easy to see
$$\tan\tet_1=\fr{\sin\vphi_1}{\cos\vphi_1+h_1}.$$
Then we have
$$d\tet_1=\Big(\arctan\fr{\sin\vphi_1}{\cos\vphi_1+h_1}\Big)'d\vphi_1=\fr{1+h_1\cos\vphi_1}{1+2h_1\cos\vphi_1+h_1^2}d\vphi_1.$$
Note that $0\leq\vphi_1\leq\pi$ and $0<h_1<1$, then $0<\tet_1<\pi$. Therefore the right side of \eqref{e:3omb} is bounded by
\Bes
\begin{split}
&\int_{\tet_1=0}^{\pi}\cdots\int_{\tet_{n-2}=0}^{\pi}\int_{\tet_{n-1}=0}^{2\pi}|\Om(x(\tet))||J(n,\tet)|
\fr{(1+2\cos\vphi_1h_1+h_1^2)^{n/2}}{1+h_1\cos\vphi_1}d\tet_{n-1}\cdots d\tet_1\\
&\leq 2^{n-1}\int_{\tet_1=0}^{\pi}\cdots\int_{\tet_{n-2}=0}^{\pi}\int_{\tet_{n-1}=0}^{2\pi}|\Om(x(\tet))||J(n,\tet)|
d\tet_{n-1}\cdots d\tet_1\\
&=2^{n-1}\int_{\S^{n-1}}|\Om(x)|d\si(x),
\end{split}
\Ees
where in the first inequality we use
$$\fr{1+2h_1\cos\vphi_1+h_1^2}{1+h_1\cos\vphi_1}\leq2$$
and $0<h_1<1$. Therefore we finish the proof of \eqref{e:5Omb}.
\end{proof}
\begin{remark} By Theorem \ref{t:5equ}, when applying the $L^1$-Dini condition, one may use its definition in Definition \ref{d:5L1} or Definition \ref{d:5L11} according to the request of application.
\end{remark}

The $L^r_{\alp}$-Dini condition that we introduce in Definition \ref{d:5Lalp} is defined by rotation. It is natural to consider the translation version.
\begin{definition}\label{d:5Lalp2}
Let $\Om$ satisfy \eqref{e:Mo1}, $1\le s\le\infty$ and $0<\alp<n$. We say that $\Om$ satisfies the $L^{s}_{\alp}$-Dini condition if

(i)\ $\Om\in L^{s}(\S^{n-1})$;

(ii)\ $\int_0^1\fr{\tilde{\om}_1(\del)}{\del^{1+\alp}}d\del<\infty$, where
$\tilde{\om}_1$ is defined by \eqref{defin2}.
\end{definition}

By using the similar way that we prove Theorem \ref{t:5equ}, we have
the following result.
\begin{theorem}\label{t:5equ2}
Let $s\geq1$ and $0<\alp<n$. $L^s_{\alp}$-Dini conditions defined respectively in Definition \ref{d:5Lalp} and Definition \ref{d:5Lalp2} are equivalent.
\end{theorem}

\section{Proof of Theorem \ref{t:main}}\label{5sec:main}

In this section we give the proof of Theorem \ref{t:main}.
Suppose $\mu$ is a signed measure on $\R^n$. For $t>0,$ let $\mu_t(E)=\mu(\fr{E}{t})$,
where $E$ is the Lebesgue measurable set in $\R^n$.

\subsection{Some elementary facts}
Let us begin with some elementary facts.

\begin{lemma}\label{e:5mea}
Let $\mu$ be a signed measure on $\R^n$. Suppose $E$ is the $\mu_t$ measurable set. Then
$$|\mu_t|(E)=|\mu|_t(E).$$
\end{lemma}
\begin{proof}
Since $\mu$ is a signed measure on $\R^n$, by the Hahn decomposition (see \cite{f}), there exists a positive set $P$ and a negative set $N$ such that $P\bigcup N=\R^n$ and $P\bigcap N={\emptyset}$. If $P'$ and $N'$ are another such pair, then $P\triangle P'(=N\triangle N')$ is null for $\mu$.
Therefore $\mu^+(E)=\mu(E\cap P)$ and $\mu^-(E)=-\mu(E\cap N)$. Since the Hahn decomposition is unique, the pair $tP$ and $tN$ can be seen as the Hahn decomposition of $\mu_t$.
Then for any $\mu_t$ measurable set $E$, we have
\Bes
\begin{split}
|\mu_t|(E)&=(\mu_t)^+(E)+(\mu_t)^-(E)=\mu_t(E\cap tP)-\mu_t(E\cap tN)\\
&=\mu\Big(\fr{1}{t}E\cap P\Big)-\mu\Big(\fr{1}{t}E\cap N\Big)\\
&=|\mu|\Big(\fr{1}{t}E\Big)=|\mu|_t(E).
\end{split}
\Ees
Hence the proof is completed.
\end{proof}

\begin{lemma}\label{l:3meas}
Let $\mu$ be a nonnegative measure defined on $\R^n$ and $\mu(\R^n)=1$. Suppose $\mu$ is absolutely continuous with respect to Lebesgue measure. Then for any
$0<\eps<1$, there exists $a_\eps$, $0<a_\eps<\infty$, such that $\mu(B(0,a_\eps))=\eps$.
\end{lemma}
\begin{proof}
Since $\mu(\R^n)=1$, there exists $M$, $0<M<\infty$, such that $\mu(B(0,M))\geq\eps$.

Set $A_\eps=\{r:\mu(B(0,r))\geq\eps\}$ and denote $a_\eps=\inf\limits_{r\in A_\eps}r$. It is easy to see that $a_\eps\leq M<\infty$. We claim that $\mu(B(0,a_\eps))=\eps$.
In fact, by the definition of infimum, for any $\alp>0$, there exists a $r\in A_\eps$, which satisfies $a_\eps<r<a_\eps+\alpha$, such that $\mu(B(0,r))\geq\eps$. Hence
$$\mu(B(0,a_\eps))\geq\mu(B(0,r))-\mu(B(0,r)\backslash
B(0,a_\eps))\geq\eps-\mu(B(0,a_\eps+\alp)\backslash B(0,a_\eps)).$$
Note that $m\big(B(0,a_\eps+\alp)\backslash B(0,a_\eps)\big)\rta0$ as $\alp\rta0$. Since
$\mu$ is absolutely continuous with respected to Lebesgue measure, so $\mu(B(0,a_\eps+\alp)\backslash B(0,a_\eps))\rta0$ as $\alp\rta 0$. Hence $\mu(B(0,a_\eps))\geq\eps$.

On the other hand, by the definition of $a_\eps$, for
any $0<r<a_\eps$, we have $\mu(B(0,r))<\eps$. Note
$$\mu(B(0,a_\eps))\leq\mu(B(0,r))+\mu(B(0,a_\eps)\backslash B(0,r))<\eps+\mu(B(0,a_\eps)\backslash B(0,r)).$$
Since $\mu(B(0,a_\eps)\backslash B(0,r))\rta0$ as $r\rta a_\eps$, then $\mu(B(0,a_\eps))\leq \eps$. Therefore we finish the proof.
\end{proof}

\begin{lemma}\label{l:omega_est}
Let $0\leq\alp<n$ and $r=\fr{n}{n-\alp}$. For a fixed $\lambda>0$, we have
\begin{equation}
\lambda^r m\Big(\Big\{x\in\R^n:\frac{|\Omega(x)|}{|x|^{n-\alp}}>\lambda\Big\}\Big)
=\frac{1}{n}\int_{\mathbb{S}^{n-1}}|\Omega(\theta)|^rd\si(\theta).
\end{equation}
\end{lemma}
\begin{proof}
By making a polar transform,
\Bes
\begin{split}
m\Big(\Big\{x\in\R^n:\fr{|\Om(x)|}{|x|^{n-\alp}}>\lam\Big\}\Big)&=\int_{\S^{n-1}}\int_0^\infty
\chi_{\{|\Om(\tet)|/s^{n-\alp}>\lam\}}s^{n-1}dsd\si(\tet)\\
&=\int_{\S^{n-1}}\int_0^{(\fr{|\Om(\tet)|}{\lam})^{\fr{1}{n-\alp}}}s^{n-1}dsd\si(\tet)\\
&=\fr{1}{n\cdot\lam^r}\int_{\S^{n-1}}|\Om(\tet)|^rd\si(\tet).
\end{split}
\Ees
\end{proof}

\begin{lemma}\label{l:3tsup}
Let $\mu$ be a absolutely continuous signed measure on $\R^n$
with respect to Lebesgue measure and $|\mu|(\R^n)<\infty$. Suppose $\Om$ satisfies \eqref{e:Mo1}, \eqref{e:Mo2} and the $L^1$-Dini condition. For any $\lambda>0$, we have
\begin{equation}\label{e:3weak}
\lambda m(\{x\in \mathbb{R}^n:|T_\Omega \mu(x)|>\lambda\})\leq C|\mu|(\R^n)
\end{equation}
where the constant $C$ only depends on $\Om$ and the dimension.
\end{lemma}

\begin{proof}
Since $\mu$ is a absolutely continuous signed measure on $\R^n$
with respect to Lebesgue measure and $|\mu|(\R^n)<\infty$, by the Radon-Nikodym's theorem (see \cite{f}), there exists a integrable function $f$ such that
$d\mu(x)=f(x)dx.$ Therefore we have
$$T_{\Om}\mu(x)=T_{\Om}f(x).$$
Now the rest of the proof can be found in the book \cite{L}. By carefully examining the proof there, the weak (1,1) bound
in \eqref{e:3weak} is $C(\|\Om\|_1+\int_0^1\fr{{\om}_1(s)}{s}ds)$.
\end{proof}
\subsection{Key lemma}
Now we give a lemma which plays a key role in the proof of Theorem \ref{t:main}.

\begin{lemma}\label{l:mu}
Let $\mu$ be an absolutely continuous signed measure with respect to  Lebesgue measure on $\R^n$ and $|\mu|(\R^n)<+\infty$. Suppose $\Om$ satisfies \eqref{e:Mo1}, \eqref{e:Mo2} and the $L^1$-Dini condition. Let $T_\Omega$ be defined by \eqref{singular with mea}. Then we have
\begin{equation}
\lim\limits_{t\rightarrow0_+}\lambda m(\{x\in \mathbb{R}^n:|T_\Omega \mu_t(x)|>\lambda\})= \frac{1}{n}\|\Omega\|_{1}|\mu(\mathbb{R}^n)|
\end{equation}
for any $\lambda>0$.
\end{lemma}
\begin{proof}
Without loss of generality,  we may assume $|\mu|(\R^n)=1$. Let $\delta$ is
small enough such that $0<\delta\ll1$.  For any fixed $\lambda>0$, choose $\eps$ such that $0<\eps\leq\frac{1}{2}\delta\lambda$. By Lemma \ref{l:3meas}, there exists an $a_\eps$ with $0<a_\eps<\infty$, such that $|\mu|(B(0,a_\eps))=1-\eps$.
Set $\eps_t=a_\eps\cdot t$, by Lemma \ref{e:5mea} we have $$|\mu_t|(B(0,\eps_t))=|\mu|_t(B(0,\eps_t))=1-\eps.$$
Let $\eta>\eps_t$. For $x\in B(0,\eta)^c$ and $y\in B(0,\eps_t)$, we can choose the minimal positive constant $\tau$ which satisfies
\Be\label{e:5xy}
\fr{1-\tau}{|x|^n}\leq\fr{1}{|x-y|^n}\leq\fr{1+\tau}{|x|^n}.
\Ee
Then $\tau\rta0_+$ as $t\rta0_+$.

Define $d\mu_t^1(x)=\chi_{B(0, \eps_t)}(x)d\mu_t(x)$ and $d\mu_t^2(x)=\chi_{B(0, \eps_t)^c}(x)d\mu_t(x)$, where $\chi_E$ is the
characteristic function of $E$. Hence we have
\[|T_\Om\mu^1_t(x)|-|T_\Om\mu^2_t(x)|\leq|T_\Om\mu_t(x)|\leq |T_\Om\mu_t^1(x)|+|T_\Om\mu_t^2(x)|.\]
For any given $\lambda>0$, let
\[F_\lam^t=\{x\in\R^n:|T_\Om\mu_t(x)|>\lam\},\]
\[F_{1, \lam}^t=\{x\in\R^n:|T_\Om\mu_t^1(x)|>\lam\}\]
and
\[F_{2, \lam}^t=\{x\in\R^n:|T_\Om\mu_t^2(x)|>\lam\}.\]

Since $\Om$ satisfies the $L^1$-Dini condition, by Lemma \ref{l:3tsup}, $T_\Om$ is of weak
type (1,1). Therefore
\begin{equation}\label{e:5F2}
\begin{split}
m(F_{2, \del\lam}^t)&=m(\{x\in\R^n:|T_\Om\mu_t^2(x)|>\del\lam\})\leq \frac{C}{\del\lam}|\mu_t^2|(\R^n)\\
&=\frac{C}{\del\lam}|\mu_t|(B(0, \eps_t)^c)\leq\frac{C\eps}{\del\lam}.
\end{split}
\end{equation}

Since $F_{1, (1+\del)\lam}^t\subset F_{2, \del\lam}^t\cup F_\lam^t$ and $F_\lam^t\subset F_{2, \del\lam}^t\cup
F^t_{1, (1-\del)\lam}$, by \eqref{e:5F2} we have the following estimate
\Be\label{e:del}
-\frac{C\eps}{\del\lam}+m(F_{1, (1+\del)\lam}^t)\leq m(F_\lam^t)\leq
 \frac{C\eps}{\del\lam}+m(F_{1, (1-\del)\lam}^t).
\Ee
By the choice of $\eps$ and $\del$, $m(F_{1, (1+\del)\lam}^t)$ and $m(F_{1, (1-\del)\lam}^t)$ may approximate to $m(F_{\lam}^t)$ as $t\rta 0_+$ by \eqref{e:del}.
It is easy to see that
\[m(F_{1, (1+\del)\lam}^t)-\om_n\eta^n\leq m(F_{1, (1+\del)\lam}^t\cap B(0, \eta)^c)\leq m(F_{1, (1+\del)\lam}^t)\]
where $\om_n$ is the Lebesgue measure of unit ball in $\R^{n}$. Therefore we conclude that $m(F_{1, (1+\del)\lam}^t\cap B(0, \eta)^c)$ approximates to
$m(F_{1, (1+\del)\lam}^t)$ as $\eta\rta 0_+$.
Similarly, $m(F_{1, (1-\del)\lam}^t\cap B(0, \eta)^c)$ approximates to
$m(F_{1, (1-\del)\lam}^t)$ as $\eta\rta 0_+$.

Now we split $T_{\Om}\mu_t^1(x)$ into two parts:
$$T_{\Om}\mu_t^1(x)=\lim_{\eps'\rta0_+}\int_{|x-y|>\eps'}\fr{\Om(x)}{|x|^n}d\mu_t^1(y)+
\lim_{\eps'\rta0_+}\int_{|x-y|>\eps'}\Big(\fr{\Om(x-y)}{|x-y|^n}-\fr{\Om(x)}{|x|^n}\Big)d\mu_t^1(y).$$
Using the triangle inequality, we have
\begin{equation}\label{e:M_Om_de}
\begin{split}
&\ \ \ \ \bigg|\int_{|x-y|>\eps'}\fr{\Om(x)}{|x|^n}d\mu_t^1(y)\bigg|-
\int_{|x-y|>\eps'}\Big|\fr{\Om(x-y)}{|x-y|^n}-\fr{\Om(x)}{|x|^n}\Big|d|\mu_t^1|(y)\\
&\leq\Big|\int_{|x-y|>\eps'}\fr{\Om(x-y)}{|x-y|^n}d\mu_t^1(y)\Big|\\
&\leq\Big|\int_{|x-y|>\eps'}\fr{\Om(x)}{|x|^n}d\mu_t^1(y)\Big|+
\int_{|x-y|>\eps'}\Big|\fr{\Om(x-y)}{|x-y|^n}-\fr{\Om(x)}{|x|^n}\Big|d|\mu_t^1|(y).
\end{split}
\end{equation}
Denote
\[G_t:=\bigg\{x\in B(0, \eta)^c:\lim\limits_{\eps'\rta 0_+}\int_{|x-y|>\eps'}\Big|\fr{\Om(x)}{|x|^n}-\fr{\Om(x-y)}
{|x-y|^n}\Big|d|\mu^1_t|(y)\geq2\del\lam\bigg\}. \]
Since
$$\bigg|\fr{\Om(x-y)}{|x-y|^n}-\fr{\Om(x)}{|x|^n}\bigg|\leq
\fr{|\Om(x-y)-\Om(x)|}{|x-y|^n}+
|\Om(x)|\Big|\fr{1}{|x-y|^n}-\fr{1}{|x|^n}\Big|,$$
we have
$$G_t\subset G_{t,1}\cap G_{t,2},$$
where
$$G_{t,1}:=\bigg\{x\in B(0, \eta)^c:\lim\limits_{\eps'\rta 0_+}
\int_{|x-y|>\eps'}\fr{|\Om(x-y)-\Om(x)|}{|x-y|^n}d|\mu^1_t|(y)
\geq\del\lam\bigg\}$$
and
$$G_{t,2}:=\bigg\{x\in B(0, \eta)^c:\lim\limits_{\eps'\rta 0_+}
\int_{|x-y|>\eps'}|\Om(x)|\Big|\fr{1}{|x-y|^n}-\fr{1}{|x|^n}\Big|d|\mu^1_t|(y)
\geq\del\lam\bigg\}.$$

Consider $G_{t,1}$ firstly.
If $x\in B(0,\eta)^c$ and $y\in B(0,\eps_t)$, then $|x|>|y|$ and
$\fr{1}{|x-y|^n}\leq\fr{1+\tau}{|x|^n}$ by \eqref{e:5xy}. Using Chebychev's inequality,
Fubini's theorem and making a polar transform, we have
\Bes
\begin{split}
m(G_{t,1})&\leq m\bigg(\bigg\{x\in B(0, \eta)^c:\int_{\R^n}\fr{|\Om(x)-\Om(x-y)|}{|x|^n}d|\mu_t^1|(y)
\geq\fr{\del\lam}{1+\tau}\bigg\}\bigg)\\
&\leq\fr{1+\tau}{\lam\del}\int_{B(0, \eta)^c}\int_{\R^n}\frac{|\Om(x-y)-\Om(x)|}{|x|^n}d|\mu^1_t|(y)dx\\
&=\fr{1+\tau}{\lam\del}\int_{\R^n}\int_{B(0, \eta)^c}\fr{|\Om(x-y)-\Om(x)|}{|x|^n}dxd|\mu_t^1|(y)\\
&=\frac{1+\tau}{\lam\del}\int_{\R^n}\int_{\eta}^{+\infty}\int_{\mathbb{S}^{n-1}}{\Big|\Om(\theta-\frac{y}{r})-\Om(\theta)\Big|}
d\sigma(\theta)\cdot \frac {dr}{r} d|\mu_t^1|(y)
\end{split}
\Ees
By  Theorem \ref{t:5equ}, the $L^1$-Dini condition in Definition \ref{d:5L11} and Definition \ref{d:5L1} are equivalent. So in the following we use the $L^1$-Dini condition in Definition \ref{d:5L11}.
Set $A(r):=\int_0^{r}\fr{\tilde{\om}_1(s)}{s}ds$. Since $\Om$ satisfies the $L^1$-Dini condition, we have $A(r)\rta0$ as $r\rta0_+$. Therefore
\Be
\begin{split}
m(G_{t,1})&\leq \fr{(1+\tau)}{\lam\del}\int_{\R^n}\int_{\eta}^{+\infty}\fr{\tilde{\om}_1(|y|/r)}{ r}drd|\mu_t^1|(y)\\
&=\fr{(1+\tau)}{\lam\del}\int_{\R^n}\int_{0}^{|y|/\eta}\fr{\tilde{\om}_1(s)}{s}dsd|\mu_t^1|(y)\\
&\leq\fr{(1+\tau)}{\del\lam}\int_0^{\eps_t/\eta}\fr{\tilde{\om}_1(s)}{s}ds\int_{\R^n}d|\mu_t^1|(y)\\
&\leq\fr{(1+\tau)}{\del\lam}A(\eps_t/\eta),
\end{split}
\Ee
where in the second equality we make a transform $|y|/r=s$.

Estimate of $m(G_{t,2})$ is similar to that of $m(G_{t,1})$. Again by using Chebychev's inequality,
Fubini's theorem, \eqref{e:5xy} and making a polar transform, we have
\begin{equation}
\begin{split}
m(G_{t,2})&\leq \fr{1}{\del\lam}\int_{B(0,\eta)^c}
\int_{\R^n}|\Om(x)|\Big|\fr{1}{|x|^n}-\fr{1}{|x-y|^n}\Big|d|\mu_t^1|(y)dx\\
&\leq\fr{1}{\del\lam}\int_{\R^n}\int_{B(0,\eta)^c}|\Om(x)|\fr{(1+\tau)n|y|}{|x|^{n+1}}dxd|\mu_t^1|(y)\\
&\leq\fr{(1+\tau)n}{\del\lam}\|\Om\|_1\int_{\R^n}\int_{\eta}^\infty\fr{dr}{r^2}|y|d|\mu_t^1|(y)\\
&\leq\fr{(1+\tau)n\eps_t}{\del\lam\eta}\|\Om\|_1|\mu_t^1|(\R^n)\\
&\leq\fr{(1+\tau)n\eps_t}{\del\lam\eta}\|\Om\|_1,
\end{split}
\end{equation}
where in the fourth inequality we use $d\mu_{t}^1=\chi_{B(0,\eps_t)}d\mu_{t}$.
Therefore combining these estimates for $G_{t,1}$ and $G_{t,2}$, we have
\Be\label{e:G_t}
m(G_t)\leq m(G_{t,1})+m(G_{t,2})\leq\fr{(1+\tau)}{\del\lam}A(\eps_t/\eta)+
\fr{(1+\tau)n\eps_t}{\del\lam\eta}\|\Om\|_1.
\Ee
It is easy to see that
\Bes
\begin{split}
m(\{x\in B(0,\eta)^c\cap G_t^c:&|T_\Om\mu_t^1(x)|>\lam\})
\leq m(\{F_{1,\lam}^t\cap B(0,\eta)^c\})\\
&\leq m(\{x\in B(0,\eta)^c\cap G_t^c: |T_\Om\mu_t^1(x)|>\lam\})+m(G_t).
\end{split}
\Ees
So if $x\in B(0, \eta)^c\cap G_t^c$, by the definition of $G_t$ and \eqref{e:M_Om_de},
$$\fr{|\Om(x)|}{|x|^n}|\mu_t^1(\R^n)|-2\del\lam\leq|T_\Om\mu_t^1(x)|\leq
\fr{|\Om(x)|}{|x|^n}|\mu_t^1(\R^n)|+2\del\lam.$$
Therefore we have
\begin{equation}\label{e:5Tsp}
\begin{split}
 \Big\{&x\in B(0,\eta)^c\cap G_t^c:|T_\Om\mu_t^1(x)|>(1-\del)\lam\Big\}\\
 &\subset\bigg\{x\in B(0,\eta)^c\cap
G_t^c:\fr{|\Om(x)|}{|x|^n}|\mu_t^1(\R^n)|>(1-3\del)\lam\bigg\}
\end{split}
\end{equation}
and
\begin{equation}\label{e:5Tsp1}
\begin{split}
 \Big\{&x\in B(0,\eta)^c\cap G_t^c:\  |T_\Om\mu_t^1(x)|>(1+\del)\lam\Big\}\\
 &\supset\bigg\{x\in B(0,\eta)^c\cap G_t^c:
\fr{|\Om(x)|}{|x|^n}|\mu_t^1(\R^n)|>(1+3\del)\lam\bigg\}.
\end{split}
\end{equation}
By the definition of $\mu_t^1$,
$$|\mu_t^1(\R^n)|=|\mu(\R^n)-\mu_t(B(0,\eps_t)^c)|.$$
Note that $|\mu_t(B(0,\eps_t)^c)|\leq|\mu_t|(B(0,\eps_t)^c)\leq\eps$, so we have
\[|\mu(\R^n)|-\eps<|\mu_t^1(\R^n)|\leq|\mu(\R^n)|+\eps.\]
Using \eqref{e:G_t}, \eqref{e:5Tsp},
\eqref{e:5Tsp1} and Lemma \ref{l:omega_est} with $\alp=0$, we have
\begin{equation}\label{e:5tinf}
\begin{split}
&\ \ \ \ m(F_{1,(1+\del)\lam}^t)\\&\geq m(\{x\in B(0,\eta)^c\cap G_t^c:|T_{\Om}\mu_t^1(x)|>(1+\del)\lam\})\\
&\geq m\Big(\Big\{x\in B(0,\eta)^c\cap G_t^c:\fr{|\Om(x)|}{|x|^n}|\mu_t^1(\R^n)|\geq(1+3\del)\lam\Big\}\Big)\\
&\geq m\Big(\Big\{x\in\R^n:\fr{|\Om(x)|}{|x|^n}|\mu_t^1(\R^n)|>(1+3\del)\lam\Big\}\Big)
-\om_n\eta^n-m(G_t)\\
&\geq\fr{\|\Om\|_1}{n}\cdot\fr{|\mu(\R^n)|-\eps}{(1+3\del)\lam} -\om_n\eta^n-
\fr{(1+\tau)}{\del\lam}A(\fr{\eps_t}{\eta})-\fr{(1+\tau)n\eps_t}{\del\lam\eta}\|\Om\|_1
\end{split}
\end{equation}
and
\begin{equation}\label{e:5tinf1}
\begin{split}
&\ \ \ \ m(F_{1,(1-\del)\lam}^t)\\&\leq m(\{x\in B(0,\tau)^c\cap G_t^c:|T_\Om\mu_t^1(x)|>(1-\del)\lam\})+m(B(0,\eta))+m(G_t)\\
&\leq m\Big(\Big\{x\in\R^n:\fr{|\Om(x)|}{|x|^n}|\mu_t^1(\R^n)|>(1-3\del)\lam\Big\}\Big)+\om_n\eta^n
+m(G_t)\\
&\leq\fr{\|\Om\|_1}{n}\cdot\fr{|\mu(\R^n)|+\eps}{(1-3\del)\lam}+\om_n\eta^n
+\fr{(1+\tau)}{\del\lam}A(\fr{\eps_t}{\eta})+\fr{(1+\tau)n\eps_t}{\del\lam\eta}\|\Om\|_1.
\end{split}
\end{equation}
Here $\om_n$ is the volume of unit ball in $\R^n$. Combining the above estimates \eqref{e:5tinf}, \eqref{e:5tinf1} and \eqref{e:5F2}, we have
\begin{equation*}
\begin{split}
m(F_\lam^t)&\geq m(F_{1,(1+\del)\lam}^t)-m(F_{2,\del\lam}^t)\\
&\geq
\fr{\|\Om\|_1}{n}\fr{|\mu(\R^n)|-\eps}{(1+3\del)\lam}-\om_n\eta^n-
\fr{(1+\tau)}{\del\lam}A(\fr{\eps_t}{\eta})-\fr{(1+\tau)n\eps_t}{\del\lam\eta}\|\Om\|_1-\fr{C\eps}{\del\lam}
\end{split}
\end{equation*}
and
\begin{equation*}
\begin{split}
m(F_\lam^t)&\leq m(F_{1,(1-\del)\lam}^t)+m(F_{2,\del\lam}^t)\\
&\leq\fr{\|\Om\|_1}{n}\fr{|\mu(\R^n)|+\eps}{(1-3\del)\lam}
+\om_n\eta^n+\fr{(1+\tau)}{\del\lam}A(\fr{\eps_t}{\eta})+\fr{(1+\tau)n\eps_t}{\del\lam\eta}\|\Om\|_1+\fr{C\eps}{\del\lam}.
\end{split}
\end{equation*}
Let $t\rta 0_+$, then $\eps_t\rta 0_+$ and $\tau\rta0_+$. So $A(\fr{\eps_t}{\eta})\rta 0_+$. Thus we obtain
\Bes
\liminf_{t\rta0_+}m(F_\lam^t)\geq
\fr{\|\Om\|_1}{n}\fr{|\mu(\R^n)|-\eps}{(1+3\del)\lam}-\om_n\eta^n-\fr{C\eps}{\del\lam}
\Ees
and
\Bes
\limsup_{t\rta0_+}m(F_\lam^t)\leq
\fr{\|\Om\|_1}{n}\fr{|\mu(\R^n)|+\eps}{(1-3\del)\lam}
+\om_n\eta^n+\fr{C\eps}{\del\lam}.
\Ees
Note that $\eps\leq\frac{1}{2}\delta\lam$. Now let $\eps\rta0_+$ firstly and $\del\rta0_+$ secondly. Lastly let $\eta\rta0_+$. Then we have
\Bes
\fr{\|\Om\|_1|\mu(\R^n)|}{n\lam}\leq\liminf_{t\rta0_+}m(F_\lam^t)\leq\limsup_{t\rta0_+}m(F_\lam^t)\leq\fr{\|\Om\|_1|\mu(\R^n)|}{n\lam}.
\Ees
Thus we complete the proof.
\end{proof}

\subsection{The proof of Theorem \ref{t:main}}

We write $T_{\Om}\mu_t(x)$ as
\Be\label{e:M_Om_eq}
\begin{split}
T_{\Om}\mu_t(x)&=\lim_{\ep\rta0_+}\int_{|x-y|>\ep}\fr{\Om(x-y)}{|x-y|^n}d\mu_t(y)\\
&=\fr{1}{t^n}\lim_{\ep\rta0_+}
\int_{|\fr{x-y}{t}|>\ep}\fr{\Om\big(\fr{x}{t}-\fr{y}{t}\big)}
{|\fr{x}{t}-\fr{y}{t}|^n}d\mu\big(\fr{y}{t}\big)=\fr{1}{t^n}T_{\Om}\mu\big(\fr{x}{t}\big).
\end{split}
\Ee
Then by (\ref{e:M_Om_eq}), we have
\Bes
\begin{split}
m(\{x\in\R^n:|T_\Om\mu_t(x)|>\lam\})&=
m\Big(\Big\{x\in\R^n:\fr{1}{t^n}|T_\Om\mu\big(\fr{x}{t}\big)|>\lam\Big\}\Big)\\&=t^nm(\{x\in\R^n:|T_\Om\mu(x)|>\lam t^n\}).
\end{split}
\Ees
Applying Lemma \ref{l:mu}, we get
\Bes
\begin{split}
\lim\limits_{\lam\rta 0_+}\lam m(\{x\in\R^n:|T_\Om\mu(x)|>\lam\})&=\lim\limits_{t\rta 0_+}\lam t^n
m(\{x\in\R^n:|T_\Om\mu(x)|>\lam t^n\})\\
&=\lim\limits_{t\rta 0_+}\lam m(\{x\in\R^n:|T_\Om\mu_t(x)|>\lam\})\\
&=\fr{1}{n}\|\Om\|_{1}|\mu(\R^n)|.
\end{split}
\Ees
Hence we complete the proof of Theorem \ref{t:main}.
$\hfill{} \Box$

\section{Proof of Theorem \ref{t:2}}\label{s:54}
In this section, we give the proof of Theorem \ref{t:2}. The proof is quite similar to that of Theorem \ref{t:main}. So we shall be brief and  only indicate necessary modifications here. We first set up a result for $T_{\Om,\alp}$ which is similar to Lemma \ref{l:mu}.

\begin{lemma}\label{l:mualp}
Set $0<\alp<n$ and $r=\fr{n}{n-\alp}$. Let $\mu$ be an absolutely continuous signed measure with respect to Lebesgue measure on $\R^n$ and $|\mu|(\R^n)<+\infty$. Suppose $\Om$ satisfies \eqref{e:Mo1}, \eqref{e:Mo2} and the $L^r_\alp$-Dini condition. Then we have
\begin{equation}
\lim\limits_{t\rightarrow0_+}\lambda^{r} m(\{x\in \mathbb{R}^n:|T_{\Omega,\alp} \mu_t(x)|>\lambda\})= \frac{1}{n}\|\Omega\|^{r}_{r}|\mu(\mathbb{R}^n)|^r.
\end{equation}
for any $\lambda>0$.
\end{lemma}
\begin{proof}
The proof is similar to that of Lemma \ref{l:mu}. Choose the same constants $\del$, $\eps$, $a_\eps$ and $\eps_t$ as we do in the proof of Lemma \ref{l:mu}.
For the constant $\tau$ we choose the minimal constant such that
$$\fr{1-\tau}{|x|^{n-\alp}}\leq\fr{1}{|x-y|^{n-\alp}}\leq\fr{1+\tau}{|x|^{n-\alp}}.$$
Since $T_{\Om,\alp}$ is bounded from $L^1(\R^n)$ to $L^{\fr{n}{n-\alp},\infty}$ (see Page.224 in \cite{CWW}), we can get the similar estimate in \eqref{e:5F2}.
For the estimate similar to $m(G_{t,1})$, by Theorem \ref{t:5equ2}, we use the equivalent $L^r_{\alp}$-Dini condition in Definition \ref{d:5Lalp2}.
In the estimate similar to \eqref{e:5tinf} and (\ref{e:5tinf}$'$), we can use Lemma \ref{l:omega_est} with $0<\alp<n$. Proceeding the proof as we do in the proof of Lemma \ref{l:mu}, we can finish the proof of Lemma \ref{l:mualp}.
\end{proof}

\noindent
\emph{\textbf{Proof of Theorem \ref{t:2}.}}
As we have done in the last part of  section \ref{5sec:main},
we can establish the following dilation
property of $T_{\Om,\alp}$ which is similar to \eqref{e:M_Om_eq}:
$$T_{\Om,\alp}\mu_t(x)=\fr{1}{t^{n-\alp}}T_{\Om,\alp}\mu(\fr{x}{t}).$$
By using above equality and Lemma \ref{l:mualp}, we have
\begin{equation*}
\begin{split}
&\ \ \ \ \lim\limits_{\lam\rta0_+}\lam^{r}m(\{x\in\R^n:|T_{\Om,\alp}\mu(x)|>\lam\})\\
&=\lim\limits_{t\rta0_+}(\lam t^{n-\alp})^{r}m(\{x\in\R^n:|T_{\Om,\alp}\mu(x)|>\lam t^{n-\alp}\})\\
&=\lim\limits_{t\rta0_+}\lam^{r}m\Big(\Big\{x\in\R^n:|T_{\Om,\alp}\mu(\fr{x}{t})|>\lam t^{n-\alp}\Big\}\Big)\\
&=\lim\limits_{t\rta0_+}\lambda^{r} m(\{x\in \mathbb{R}^n:|T_{\Omega,\alp} \mu_t(x)|>\lambda\})=\frac{1}{n}\|\Omega\|^{r}_{{r}}|\mu(\mathbb{R}^n)|^r.
\end{split}
\end{equation*}
Hence we complete the proof of Theorem \ref{t:2}.
$\hfill{} \Box$
\subsection*{Acknowledgment}
The authors would like to express their deep gratitude to the referee for his/her very careful reading, important comments and valuable suggestions.

\vskip1cm
\bibliographystyle{amsplain}

\end{document}